\documentclass[a4apper,10pt]{amsart}

\usepackage{latexsym}
\usepackage{amsmath}
\usepackage{amsthm}
\usepackage{amssymb}
\usepackage{amsfonts}
\usepackage{amsxtra}
\usepackage{setspace}
\usepackage[all]{xypic}

\newtheorem{thm}{Theorem}

\newtheorem{cor}[thm]{Corollary}
\newtheorem{lem}[thm]{Lemma}

\theoremstyle{definition}
\newtheorem{defn}[thm]{Definition}

\newtheorem{example}[thm]{Example}
\newtheorem*{Main}{Theorem}
\numberwithin{thm}{section}

\newcommand{\C}{\mathbb{C}}

\renewcommand{\Im}{\textrm{Im}}

\newcommand{\Fun}{\textrm{Fun}}

\begin{document}
\title{Factorization in generalized Calogero-Moser spaces}
\author{Gwyn Bellamy}
\address{School of Mathematics and Maxwell Institute for Mathematical Sciences, University of Edinburgh, James Clerk Maxwell Building, Kings Buildings, Mayfield Road, Edinburgh EH9 3JZ, Scotland}
\email{G.E.Bellamy@sms.ed.ac.uk}
\maketitle
\begin{abstract}
\noindent Using a recent construction of Bezrukavnikov and Etingof, \cite{108}, we prove that there is a factorization of the Etingof-Ginzburg sheaf on the generalized Calogero-Moser space associated to a complex reflection group. In the case $W = S_n$, this confirms a conjecture of Etingof and Ginzburg, \cite{1}.
\end{abstract}

\section{Introduction}
In this paper we apply a recent construction of Bezrukavnikov and Etingof, \cite{108}, to the study of the centres of the rational Cherednik algebras at $t = 0$. The affine varieties $X_{\mathbf{c}}(W)$, corresponding to these centres are called the \textit{generalized Calogero-Moser spaces} and are known to influence the representation theory of the algebras. We show that there exists an isomorphism of schemes
\begin{equation}\label{eq:intro}
\Phi : \pi^{-1}_{W}(b) \stackrel{\sim}{\longrightarrow} \pi^{-1}_{W_b}(0),
\end{equation}
where $\pi_W : X_{\mathbf{c}}(W) \twoheadrightarrow \mathfrak{h}/W$ and $W_b$ is a parabolic subgroup of $W$ associated to the orbit $b \in \mathfrak{h}/W$ (for precise definitions see Section \ref{sec:definitions}). In order to relate the representation theory of the algebras $H_{0,\mathbf{c}}$ to the varieties $X_{\mathbf{c}}$, Etingof and Ginzburg introduced the coherent sheaf $\mathcal{R}[W]$, defined by $\Gamma (X_{\mathbf{c}},\mathcal{R}[W]) = H_{0,\mathbf{c}} \mathbf{e}$. Our main result describes the pushforward of $\mathcal{R}[W]_{|_{\pi^{-1}_{W}(b)}}$ by $\Phi$.
\begin{Main}
On $\pi^{-1}_{W_b}(0)$ there is an isomorphism of $W$-equivariant sheaves
\begin{displaymath}
\Phi_* \left( \mathcal{R} [ W ]_{|_{\pi_W^{-1}(b)}} \right) \simeq \texttt{Ind}_{\, W_b}^{\, W} \mathcal{R} [ W_b ]_{|_{\pi_{W_b}^{-1}(0)}}
\end{displaymath}
\end{Main}
\noindent In particular, this theorem proves that the sheaf $\mathcal{R}[W]$ on the Calogero-Moser space associated to the symmetric group factorizes as conjectured by Etingof and Ginzburg, \cite[page 319]{1}. 

\section{The rational Cherednik algebra}\label{sec:definitions}

\subsection{Definitions and notation}\label{subsection:defns}

Let $W$ be a complex reflection group, $\mathfrak{h}$ its reflection
representation over $\C$ with rank $\mathfrak{h} = n$, and $\mathcal{S}$ the set of all complex reflections in $W$. The idempotent in $\C W$ corresponding to the trivial representation will be denoted $\mathbf{e}_W$. Let $( \cdot, \cdot ) : \mathfrak{h} \times \mathfrak{h}^* \rightarrow \C$ be the natural pairing defined by $(y,x) = x(y)$. For $s \in S$, fix $\alpha_s \in \mathfrak{h}^*$ to be a basis of the one dimensional space $\Im(s - 1)|_{\mathfrak{h}^*}$ and $\alpha_s^{\vee} \in \mathfrak{h}$ a basis of the one dimensional space $\Im(s - 1)|_{\mathfrak{h}}$ such that $\alpha_s(\alpha_s^\vee) = 2$. Choose $\mathbf{c} : \mathcal{S} \rightarrow \C$ to be a $W$-invariant function and $t$ a complex number. The \textit{rational Cherednik algebra}, $H_{t,\mathbf{c}}(W)$, as introduced by Etingof and Ginzburg, \cite[page 250]{1}, is the quotient of the skew group algebra of the tensor algebra, $T(\frak{h} \oplus \frak{h}^*) \rtimes W$, by the ideal generated by the relations

\begin{equation}\label{eq:rel}
[x_1,x_2] = 0 \qquad [y_1,y_2] = 0 \qquad [x_1,y_1] = t (y_1,x_1) - \sum_{s \in S} \mathbf{c}(s) (y_1,\alpha_s)(\alpha_s^\vee,x_1) s 
\end{equation}
\noindent $\forall x_1,x_2 \in \mathfrak{h}^* \textrm{ and } y_1,y_2 \in \mathfrak{h}$.\\

\noindent Since there is an isomorphism $H_{\lambda t,\lambda \mathbf{c}}(W) \cong H_{t,\mathbf{c}}(W)$ for any $\lambda \in \C^*$, we can restrict ourselves to considering the cases $t = 0$ or $1$. 

\subsubsection{Parabolic subgroups}

For a point $b \in \mathfrak{h}$, the stabilizer subgroup of $W$ with respect to $b$ will be denoted $W_b$. By a theorem of Steinberg, \cite[Theorem 1.5]{8}, $W_b$ is itself a complex reflection group. If $(\mathfrak{h}^{*W_b})^\perp$ denotes the vector subspace of $\mathfrak{h}$ consisting of all vectors $y$ in $\mathfrak{h}$ such that $x(y) = 0 \quad \forall y \in \mathfrak{h}^{*W_b}$ then $\mathfrak{h} = \mathfrak{h}^{W_b} \oplus (\mathfrak{h}^{*W_b})^\perp$ is a decomposition of $\mathfrak{h}$ as a $W_b$-module. Note that $(\mathfrak{h}^{*W_b})^\perp$ is a faithful reflection representation of $W_b$ of minimal rank. 

\subsubsection{Centralizer algebras}

We recall the centralizer construction described in \cite[3.2]{108}. Let $A$ be an $\C$-algebra equipped with a homomorphism $H \longrightarrow A^{\times}$, where $H$ is a finite group. Let $G$ be another finite group such that $H$ is a subgroup of $G$. The algebra $C(G,H,A)$ is defined to be the centralizer of $A$ in the right $A$-module $P := \Fun_H(G,A)$ of $H$-invariant, $A$-valued functions on $G$. By making a choice of left coset representatives of $H$ in $G$, $C(G,H,A)$ is realized as the algebra of $|G/H|$ by $|G/H|$ matrices over $A$. For $w,g \in G$ and $f \in \Fun_H(G,A)$, $w \cdot f (g) := f(gw)$ defines, by linearity, an embedding $\iota : \C G \hookrightarrow C(G,H,A)$. Let $\mathbf{e}_{G} \in \C G$ and $\mathbf{e}_{H} \in \C H$ denote the idempotents corresponding to the trivial representation of $G$ and $H$ respectively, where $\C H$ is considered as a subalgebra of $A$.

\begin{lem}\label{lem:centralizeriso}
There are isomorphisms of $\C G$-$Z(A)$-bimodules
\begin{displaymath}
C(G,H,A) \cdot \iota (\mathbf{e}_G) \simeq \Fun_H(G,A \mathbf{e}_H) \simeq \texttt{Ind}_{\, H}^{\, G} \, A \mathbf{e}_H,
\end{displaymath}
where $Z(A)$ denotes the centre of $A$. Here $\C G$ acts on $C(G,H,A)$ by multiplication on the left via $\iota$ and on the left of $\Fun_H(G,A \mathbf{e}_H)$ also via $\iota$. 
\end{lem}

\begin{proof}
The second isomorphism is clear from the definition of
$\Fun_H(G,A)$. Let $\delta \in \Fun_H(G,A)$ be the function defined by
$\delta (g) = \mathbf{e}_H$, for all $g \in G$. We define a linear map $\zeta$ from $C(G,H,A) \cdot \iota (\mathbf{e}_G)$ to $\Fun_H(G,A \mathbf{e}_H)$ and a map $\eta$ in the opposite direction by
\begin{displaymath}
\zeta : \quad M \cdot \iota (\mathbf{e}_G) \, \mapsto \, M (\delta) 
\end{displaymath}
\begin{displaymath}
\eta : \quad f \, \mapsto \, \left( h( - ) \, \mapsto \, f( - ) \sum_{g \in G} h(g) \right),
\end{displaymath}
where $M \in C(G,H,A)$, $f \in \Fun_H(G,A \mathbf{e}_H)$ and $h \in
\Fun_H(G,A)$. After fixing left coset representatives of $H$ in $G$, a
direct calculation shows that $\eta$ is both a left and right inverse
to $\zeta$. The $G$-equivariance of $\zeta$ is clear since
\begin{displaymath}
g \cdot \zeta( M \iota(\mathbf{e}_G)) = g \cdot M(\delta) = \iota(g)
(M(\delta)) = (\iota(g) M)(\delta) = \zeta ( g \cdot M \iota(
\mathbf{e}_G))
\end{displaymath}
The $Z(A)$-equivariance of $\zeta$ is similarly clear. 
\end{proof}

\subsection{Completing the rational Cherednik algebra}

For each point $b \in \mathfrak{h}$, the completion, $\widehat{H}_{1,\mathbf{c}}(W)_b$, at the orbit $W \cdot b \in \mathfrak{h}/W$ of $H_{1,\mathbf{c}}(W)$ is defined in \cite[2.4]{108}. However, the notion of completion at $W \cdot b$ works equally well when $t = 0$ because $H_{0,\mathbf{c}}(W)$ can be thought of as a sheaf of algebras on the affine variety $\mathfrak{h}/W$. Therefore, if $\C[[\mathfrak{h}/W]]_b$ denotes the completion of $\C[\mathfrak{h}/W]$ at $W \cdot b$, we define the completion of $H_{0,\mathbf{c}}(W)$ at $b$ to be
\begin{equation}\label{eq:complete}
\widehat{H}_{0,\mathbf{c}}(W)_b := \C[[\mathfrak{h}/W]]_b \otimes_{\C[\mathfrak{h}/W]} H_{0,\mathbf{c}}(W)
\end{equation}
Crucially, we note that \cite[Theorem 3.2]{108} is independent of the parameter $t$ and hence can be applied to the case $t = 0$. We state it here for completeness.
\begin{thm}[\cite{108}, Theorem 3.2]\label{thm:BEiso}
Let $b \in \mathfrak{h}$, and define $\mathbf{c}'$ to be the restriction of $\mathbf{c}$ to the set $S_b$ of reflections in $W_b$. Then one has an isomorphism 
\begin{equation}\label{eq:BEiso}
\theta : \widehat{H}_{t,\mathbf{c}} (W,\mathfrak{h})_b \rightarrow C(W,W_b,\widehat{H}_{t,\mathbf{c}'}(W_b,\mathfrak{h})_0),
\end{equation}
defined by the following formulas. Suppose that $f \in P = Fun_{W_b}(W,\widehat{H}_{t,\mathbf{c}'}(W_b,\mathfrak{h})_0)$. Then
\begin{displaymath}
(\theta(u)f)(w) = f(wu),u \in W;
\end{displaymath}
for any $\alpha \in \mathfrak{h}^*$,
\begin{displaymath}
(\theta(x_{\alpha})f)(w) = (x_{w\alpha}^{(b)} + (w\alpha,b))f(w),
\end{displaymath}
where $x_{\alpha} \in \mathfrak{h}^* \subset H_{t,\mathbf{c}}(W,\mathfrak{h}),x_{w\alpha}^{(b)} \in H_{t,\mathbf{c}'}(W_b,\mathfrak{h})$; and for any $a \in \mathfrak{h}$,
\begin{displaymath}
(\theta(y_a)f)(w) = y_{wa}^{(b)}f(w) + \sum_{s \in S:s \notin W_b} \frac{2c_s}{1 - \lambda_s} \frac{\alpha_s(wa)}{x_{\alpha_s}^{(b)} + \alpha_s(b)}(f(sw) - f(w)).
\end{displaymath}
where $y_a \in \mathfrak{h} \subset H_{t,\mathbf{c}}(W,\mathfrak{h})$ and $y_a^{(b)}$ the same vector considered now as an element of $H_{t,\mathbf{c}'}(W_b,\mathfrak{h})$.
\end{thm}

\section{The Etingof-Ginzburg Sheaf}\label{sec:main}

Let $Z_{\mathbf{c}}(W)$ denote the centre of $H_{0,\mathbf{c}}(W)$ and $X_{\mathbf{c}}(W) = \textrm{maxspec}(Z_{\mathbf{c}}(W))$, the corresponding affine
variety. The space $X_{\mathbf{c}}(W)$ is called the \textit{generalized Calogero-Moser space} associated to the complex reflection group $W$ at the parameter $\mathbf{c}$. For $b \in \mathfrak{h}$, the maximal ideal of $\C[\mathfrak{h}/W]$ corresponding to $W \cdot b$ will be written $\mathfrak{m}(b)$ and the two-sided ideal of $H_{0,\mathbf{c}}(W)$ generated by the elements of $\mathfrak{m}(b)$ will be denoted $\langle \mathfrak{m}(b) \rangle$. Similarly, the maximal ideal of $\C[\mathfrak{h}/W_b]$ corresponding to $W_b \cdot p$, $p \in \mathfrak{h}$, will be written $\mathfrak{n}(p)$.\\

\noindent Let $\mathcal{A}$ denote the set of reflecting hyperplanes of $W$ in $\mathfrak{h}$ and, for each $H \in \mathcal{A}$, let $L_H \in \mathfrak{h}^*$ be a linear functional whose kernel is $H$ (e.g. $\alpha_s \in \mathfrak{h}^*$ if $s$ is a reflection about $H$). Choose homogeneous algebraically independent generators $F_1, \dots F_n$ of $\C[\mathfrak{h}]^W$ and $P_1, \dots , P_n$ of $\C[\mathfrak{h}]^{W_b}$. The following description of the Jacobian is due to Steinberg, \cite[Lemma]{4}.
\begin{equation}\label{eq:det}
\Pi_W := \det \left ( \frac{\partial F_i}{\partial x_j} \right) = k \prod_{H \in \mathcal{A}} L_H^{e_H - 1}
\end{equation}
where $e_H$ is the order of the cyclic group $W_H$ of elements of $W$ that fix $H$ pointwise and $k$ a non-zero scalar. 

\begin{lem}\label{lem:auto}
For each $b \in \mathfrak{h}$ the map $\Psi : \C[[\mathfrak{h}/W_b]]_0 \longrightarrow \C[[\mathfrak{h}/W_b]]_0$ defined by
\begin{displaymath}
P_i(\mathbf{x}) \mapsto F_i(\mathbf{x} + b) - F_i(b)
\end{displaymath}
is an automorphism.
\end{lem}

\begin{proof}
Since $F_i(\mathbf{x} + b) - F_i(b) \in \mathfrak{n}(0)$ for all $i$ there exist polynomials $Q_1 , \dots ,Q_n$ such that $F_i(\mathbf{x} + b) - F_i(b) = Q_i(P_1, \dots, P_n)$. The chain rule gives
\begin{displaymath}
D := \det \left( \frac{\partial(F_i( \mathbf{x} + b) - F_i(b))}{\partial x_j} \right) = \det  \left( \frac{\partial Q_i}{\partial P_k} \right) \det  \left( \frac{\partial P_k}{\partial x_j} \right)
\end{displaymath}
However, $D = \Pi_W(\mathbf{x} + b)$ and this gives 
\begin{displaymath}
\prod_{H \in \mathcal{A}} L_H^{e_H - 1}(\mathbf{x} + b) = \det  \left( \frac{\partial Q_i}{\partial P_k} \right) \prod_{H \in \mathcal{A} \textrm{ with } b \in H} L_H^{e_H - 1}(\mathbf{x})
\end{displaymath}
Since $L_H(\mathbf{x} + b) = L_H(\mathbf{x})$ if and only if $b \in H$, we get
\begin{displaymath}
\det  \left( \frac{\partial Q_i}{\partial P_k} \right) = \prod_{H \in \mathcal{A} \textrm{ with } b \notin H} L_H^{e_H - 1}(\mathbf{x} + b)
\end{displaymath}
and
\begin{displaymath}
\det  \left( \frac{\partial Q_i}{\partial P_k} \right) (0)  = \prod_{H \in \mathcal{A} \textrm{ with } b \notin H} L_H^{e_H - 1}(b) \neq 0.
\end{displaymath}
Hence, by \cite[Exercise 7.25]{5}, $\Psi$ is an isomorphism.
\end{proof}

As a consequence of Theorem \ref{thm:BEiso}, we have an isomorphism of quotient algebras.
\begin{cor}\label{cor:quoiso}
Let $\theta : \widehat{H}_{0,\mathbf{c}} (W,\mathfrak{h})_b \rightarrow C(W,W_b,\widehat{H}_{0,\mathbf{c}'}(W_b,\mathfrak{h})_0)$ be the isomorphism (\ref{eq:BEiso}). Then $\theta$ descends to an isomorphism
\begin{equation}\label{eq:quoiso}
\theta : \frac{H_{0,\mathbf{c}}(W,\mathfrak{h})}{\langle \mathfrak{m}(b) \rangle} \stackrel{\sim}{\longrightarrow} C \left( W,W_b,\frac{H_{0,\mathbf{c}'}(W_b,\mathfrak{h})}{\langle \mathfrak{n}(0)\rangle} \right) .
\end{equation}
\end{cor}

\begin{proof}
For $a \in \mathfrak{h}$, $\alpha \in \mathfrak{h}^*$ and $w \in W$, $(x_{w \cdot \alpha} + (w \alpha,b))(a) = (w \alpha,a) + (w \alpha,b) = (w \cdot x_{\alpha})(a + b)$. Therefore $\theta(g)(f(w)) = (w \cdot g)(\mathbf{x} + b)f(w) = g(\mathbf{x} + b)f(w)$ for all $g \in \C[\mathfrak{h}]^W \subset \C[[\mathfrak{h}]]_b$ and $f \in \Fun_{W_b}(W,\widehat{H}_c(W_b,\mathfrak{h})_0)$. Now choose $u \in W_b$, then 
\begin{displaymath}
u \cdot g( \mathbf{x} + b) = g( u^{-1} \cdot \mathbf{x} + b) = g( u^{-1} \cdot (\mathbf{x} + b)) = g(\mathbf{x} + b)
\end{displaymath}
shows that $g(\mathbf{x} + b) \in \C[\mathfrak{h}]^{W_b}$. Hence, if $g \in \mathfrak{m}(b) \lhd \C[\mathfrak{h}]^W$, then $g(\mathbf{x} + b) \in \mathfrak{n}(0) \lhd \C[\mathfrak{h}]^{W_b}$. This shows that $\theta(g)f(w) \in \mathfrak{n}(0)\widehat{H}_{t,\mathbf{c}'}(W_b)_0$ and 
\begin{equation}\label{eq:inc}
\theta(\mathfrak{m}(b)\widehat{H}_{t,\mathbf{c}} (W,\mathfrak{h})_b) \subseteq C(W,W_b,\mathfrak{n}(0)\widehat{H}_{t,\mathbf{c}'}(W_b,\mathfrak{h})_0).
\end{equation}
The ideal $\mathfrak{m}(b)$ in $\C[\mathfrak{h}]^W$ is generated by $F_1(\mathbf{x}) - F_1(b), \dots , F_n(\mathbf{x}) - F_n(b)$ and we have $\theta(F_i(\mathbf{x}) - F_i(b))f(w) = (F_i(\mathbf{x} + b) - F_i(b))f(w)$. The statement of Lemma \ref{lem:auto} is equivalent to the fact that 
\begin{displaymath}
\{ F_1(\mathbf{x} + b) - F_1(b), \dots , F_n(\mathbf{x} + b) - F_n(b)) \} \cdot \C[[\mathfrak{h}/W_b]]_0 = \mathfrak{n}(0) \C[[\mathfrak{h}/W_b]]_0.
\end{displaymath}
This, together with (\ref{eq:inc}), implies that 
\begin{displaymath}
\theta(\mathfrak{m}(b)\widehat{H}_{t,\mathbf{c}} (W,\mathfrak{h})_b) = C(W,W_b,\mathfrak{n}(0)\widehat{H}_{t,\mathbf{c}'}(W_b,\mathfrak{h})_0).
\end{displaymath}
and the isomorphism follows. 
\end{proof}

By \cite[Proposition 4.15]{1}, $\C[\mathfrak{h}]^W$ is contained in the centre of $H_{0,\mathbf{c}}(W,\mathfrak{h})$ and the embedding defines a surjective morphism $\pi_{W} : X_{\mathbf{c}}(W) \twoheadrightarrow \mathfrak{h}/W$. The algebra $Z_{0,\mathbf{c}}(W)/ \langle \mathfrak{m}(b) \rangle$ is the coordinate ring of the scheme-theoretic pull-back $\pi_{W}^{-1}(b)$. Comparing the centres of the algebras in Corollary \ref{cor:quoiso} gives an isomorphism of (non-reduced) schemes.

\begin{cor}\label{cor:geo}
For $b \in \mathfrak{h}$, there is a scheme-theoretic isomorphism
\begin{equation}\label{eq:geo}
\Phi : \pi^{-1}_{W}(b) \stackrel{\sim}{\longrightarrow} \pi^{-1}_{W_b}(0)
\end{equation}
\end{cor}

\begin{proof}
The Satake isomorphism, \cite[Theorem 3.1]{1}, is the map $Z_{0,\mathbf{c}}(W) \rightarrow \mathbf{e}_W \, H_{0,\mathbf{c}}(W,\mathfrak{h}) \, \mathbf{e}_W$ defined by $z \mapsto z \, \mathbf{e}_W$. Since $ \mathfrak{m}(b) H_{0,\mathbf{c}}(W,\mathfrak{h})$ is a centrally generated ideal in $H_{0,\mathbf{c}}(W,\mathfrak{h})$,
\begin{displaymath}
 \mathfrak{m}(b) H_{0,\mathbf{c}}(W,\mathfrak{h}) \cap  \mathbf{e}_W \, H_{0,\mathbf{c}}(W,\mathfrak{h}) \, \mathbf{e}_W = \langle \mathbf{e}_W \, \mathfrak{m}(b) \rangle,
\end{displaymath}
where the right-hand side is considered as an ideal in $\mathbf{e}_W \, H_{0,\mathbf{c}}(W,\mathfrak{h}) \, \mathbf{e}_W$. Therefore the Satake isomorphism descends to an isomorphism
\begin{equation}\label{eq:satake}
\textsf{S}_{W,b} \, : \, \frac{Z_{0,\mathbf{c}}(W)}{\mathfrak{m}(b) Z_{0,\mathbf{c}}(W)} \stackrel{\sim}{\longrightarrow} \mathbf{e}_W \left( \frac{H_{0,\mathbf{c}}(W,\mathfrak{h})}{ \langle \mathfrak{m}(b) \rangle } \right) \mathbf{e}_W .
\end{equation}
As noted in \cite[Lemma 3.1 (ii)]{108}, the isomorphism (\ref{eq:quoiso}) restricts to an isomorphism of subalgebras
\begin{equation}\label{eq:sphericaliso}
\theta \, : \, \mathbf{e}_W \left( \frac{H_{0,\mathbf{c}}(W,\mathfrak{h})}{ \langle \mathfrak{m}(b) \rangle } \right) \mathbf{e}_W \stackrel{\sim}{\longrightarrow} \mathbf{e}_{W_b} \left( \frac{H_{0,\mathbf{c}'}(W_b,\mathfrak{h})}{ \langle \mathfrak{n}(0) \rangle } \right) \mathbf{e}_{W_b},
\end{equation}
where $\theta( \, \mathbf{e}_W) = \mathbf{e}_{W_b}$. Here we have implicitly identified the spherical algebra on the right-hand side with a subalgebra of $C \left( W,W_b,\frac{H_{0,\mathbf{c}'}(W_b,\mathfrak{h})}{\langle \mathfrak{n}(0)\rangle} \right)$. It is possible, though uninformative, to give an explicit description of this identification. Combining the isomorphisms of (\ref{eq:satake}) and (\ref{eq:sphericaliso}) produces the comorphism 
\begin{equation}\label{eq:phiiso}
(\Phi^*)^{-1} = \textsf{S}^{-1}_{W_b,0} \circ \theta \circ S_{W,b} \, : \, \frac{Z_{0,\mathbf{c}}(W)}{\mathfrak{m}(b) Z_{0,\mathbf{c}}(W)}  \stackrel{\sim}{\longrightarrow} \frac{Z_{0,\mathbf{c}'}(W_b)}{\mathfrak{n}(0) Z_{0,\mathbf{c}'}(W_b)} 
\end{equation}
corresponding to $\Phi$.
\end{proof} 

Etingof and Ginzburg, \cite[page 247]{1}, introduced an important coherent sheaf on $X_{\mathbf{c}}(W)$, which we now recall.

\begin{defn}
The \textit{Etingof-Ginzburg sheaf} is the coherent sheaf $\mathcal{R}[W]$ on $X_{\mathbf{c}}(W)$ corresponding to the finitely generated $Z_{0,\mathbf{c}}(W)$-module $H_{0,\mathbf{c}}(W) \, \mathbf{e}_W$.
\end{defn}

The coordinate ring of a Zariski-open subset $U \subseteq X_{\mathbf{c}}(W)$ will be written $Z_{0,\mathbf{c}}(W)_{U}$. We now conclude;

\begin{thm}\label{thm:main}
Let $\mathcal{R} [ W  ]$ be the Etingof-Ginzburg sheaf on $X_{\mathbf{c}}(W)$ and $\mathcal{R} [ W_b ]$ the Etingof-Ginzburg sheaf on $X_{\mathbf{c}'}(W_b)$. For $b \in \mathfrak{h}/W$ we have an isomorphism of $W$-equivariant sheaves on $\pi^{-1}_{W_b}(0)$
\begin{equation}\label{eq:mainiso}
\Phi_* \left( \mathcal{R} [ W ]_{|_{\pi_{W}^{-1}(b)}} \right) \simeq \texttt{Ind}_{\, W_b}^{\, W} \mathcal{R} [ W_b ]_{|_{\pi_{W_b}^{-1}(0)}}.
\end{equation}
\end{thm}

\begin{proof}
Since $\pi^{-1}_{W}(b)$ is an affine scheme, to show that we have an
isomorphism of $W$-equivariant sheaves as stated in (\ref{eq:mainiso})
it suffices to show that the global sections are isomorphic as
$(W, Z_{0,\mathbf{c}'}(W_b) \, / \, \langle \mathfrak{n}(0) \rangle =: \textsf{Z}
)$-bimodules. Taking global sections gives
\begin{displaymath}
\Phi_* \left( \mathcal{R} [ W ]_{|_{\pi_{W}^{-1}(b)}} \right)
(\pi_{W_b}^{-1}(0)) = \left( \frac{H_{0,\mathbf{c}}(W) }{\langle
  \mathfrak{m}(b)  \rangle} \right) \mathbf{e}_W  
\end{displaymath}
and
\begin{displaymath}
\texttt{Ind}_{\, W_b}^{\, W} \mathcal{R} [ W_b ]_{|_{\pi_{W_b}^{-1}(0)}}
(\pi_{W_b}^{-1}(0)) = \texttt{Ind}_{\, W_b}^{\, W} \left(
\frac{H_{0,\mathbf{c}'}(W_b,\mathfrak{h})}{\langle
  \mathfrak{n}(0)\rangle} \right) \mathbf{e}_{W_b} .
\end{displaymath}
Thus we must show that 
\begin{displaymath}
\textsf{He} := \left( \frac{H_{0,\mathbf{c}}(W) }{\langle
  \mathfrak{m}(b)  \rangle} \right) \mathbf{e}_W  \simeq
\texttt{Ind}_{\, W_b}^{\, W} \left(
\frac{H_{0,\mathbf{c}'}(W_b,\mathfrak{h})}{\langle
  \mathfrak{n}(0)\rangle} \right) \mathbf{e}_{W_b}
\end{displaymath}
as $(W, \textsf{Z})$-bimodules. Applying the isomorphism $\theta$ (of (\ref{eq:quoiso}))
to $\textsf{He}$, and noting that the restriction of $\theta$ to $\C W$
is the map $\iota$, gives
\begin{displaymath}
\theta \, : \, \textsf{He} \simeq C \left(
W,W_b,\frac{H_{0,\mathbf{c}'}(W_b,\mathfrak{h})}{\langle
  \mathfrak{n}(0)\rangle} \right) \iota ( \mathbf{e}_W) .
\end{displaymath}
However, we now have two different actions of $\textsf{Z}$ on
$\textsf{He}$. It acts on $\textsf{He}$, viewed as global sections, via
the map $\Phi^*$, but acts on the right of $C \left(
W,W_b,\frac{H_{0,\mathbf{c}'}(W_b,\mathfrak{h})}{\langle
  \mathfrak{n}(0)\rangle} \right) \iota ( \mathbf{e}_W)$ via
$\theta^{-1}$. These two actions are the same: as stated in (\ref{eq:phiiso}),
\begin{displaymath}
\Phi^* = \textsf{S}^{-1}_{W,b} \circ \theta^{-1} \circ S_{W_b,0},
\end{displaymath}
therefore
\begin{displaymath}
h \, \mathbf{e}_W \cdot \Phi^*(z) = h \, \mathbf{e}_W \cdot \textsf{S}^{-1}_{W,b} \circ \theta \circ S_{W_b,0} (z) = h \, \mathbf{e}_W \cdot \mathbf{e}_W \, \theta^{-1}( \mathbf{e}_{W_b} \cdot z) = h \, \mathbf{e}_W
\cdot \theta^{-1}(z),
\end{displaymath} 
where $z \in \textsf{Z}$ and $h \, \mathbf{e}_W \in
\textsf{He}$ (recall that $\theta( \, \mathbf{e}_W ) = \mathbf{e}_{W_b}$, c.f. (\ref{eq:phiiso})). Noting that $\textsf{Z}$ is a subalgebra of the centre
of $H_{0,\mathbf{c}'}(W_b,\mathfrak{h}) \, / \, \langle
  \mathfrak{n}(0)\rangle$, the required bimodule isomorphism is given
  by Lemma \ref{lem:centralizeriso} where $G = W$, $H = W_b$ and $A = H_{0,\mathbf{c}'}(W_b,\mathfrak{h}) \, / \, \langle
  \mathfrak{n}(0)\rangle$.
\end{proof}

\begin{example}
In the case $W = S_n, \mathfrak{h} = \C^n$, the Calogero-Moser space $X_{\mathbf{c}}(S_n)$ has been shown by Etingof and Ginzburg, \cite[Theorem 1.23]{1}, to be isomorphic to the classical Calogero-Moser space as introduced by Kazhdan, Kostant and Sternberg and studied by Wilson, \cite{2}. It is known to be smooth for $\mathbf{c} \neq 0$ (\cite[Corollary 16.2]{1} or \cite[Proposition 1.7]{2}), therefore \cite[Theorem 1.7 (i)]{1} implies that $\mathcal{R}[S_n]$ is a vector bundle of rank $n!$ on $X_{\mathbf{c}}(S_n)$. Identifying $\C^n /S_n$ with $Sym^n(\C)$, a point of $\C^n / S_n$ has the form $n_1 x_1 + \dots + n_k x_k$, where $n_1 + \dots + n_k = n$ and $x_1, \dots , x_k \in \C$ are pairwise distinct. Given $b \in \C^n$ such that $S_n \cdot b = n_1 x_1 + \dots + n_k x_k$, the stabilizer $(S_n)_b$ is conjugate to $S_{n_1} \times \dots \times S_{n_k}$. For $W = S_n$, the isomorphism of Corollary \ref{cor:geo} induces, after factoring out nilpotent elements, an isomorphism of varieties
\begin{equation}\label{eq:spaces}
\pi^{-1}_{S_n}(b) \simeq \pi^{-1}_{S_{n_1}}(0) \times \dots \times \pi^{-1}_{S_{n_k}}(0).
\end{equation}
In \cite[Lemma 7.1]{2}, Wilson explicitly constructs an isomorphism between the subvarieties of the classical Calogero-Moser space coinciding with the varieties of (\ref{eq:spaces}). Let $\boxtimes$ denote the external tensor product of vector bundles, then Theorem \ref{thm:main} implies that
\begin{displaymath}
\Phi_* \left( \mathcal{R} [ S_n ]_{|_{\pi_{S_n}^{-1}(b)}} \right) \simeq \texttt{Ind}_{S_{n_1} \times \dots \times S_{n_k}}^{S_n} \left( \mathcal{R}[S_{n_1}]_{|_{\pi_{S_{n_1}}^{-1}(0)}} \boxtimes \, \dots \, \boxtimes \mathcal{R}[S_{n_k}]_{|_{\pi_{S_{n_k}}^{-1}(0)}} \right)
\end{displaymath}
as $S_n$-equivariant vector bundles. This confirms the conjectured factorization given in \cite[11.27]{1}.
\end{example}

\section*{Acknowledgements}

The research described here was done at the University of Edinburgh with the financial support of the EPSRC. This material will form part
of the author's PhD thesis for the University of Edinburgh. The author would like to express his gratitude to his supervisor, Professor Iain Gordon, for suggesting this problem and for his help, encouragement and patience.

\bibliographystyle{plain}
\bibliography{biblo}

\begin{thebibliography}{100}

\bibitem[BE]{108}
\textbf{R. Bezrukavnikov} and \textbf{P. Etingof}, \textit{Induction and restriction functors for rational Cherednik algebras}, \textit{arXiv:08033639}.

\bibitem[E]{5}
\textbf{D. Eisenbud}, \textit{Commutative Algebra With a View Toward Algebraic Geometry}, Springer-Verlag, New York, (1994).

\bibitem[EG]{1}
\textbf{P. Etingof} and \textbf{V. Ginzburg}, \textit{Symplectic reflection algebras, Calogero-Moser space, and deformed Harish-Chandra homomorphisms}, Invent. Math 147, 243-348 (2002).


\bibitem[KKS]{6}
\textbf{D. Kazhdan}, \textbf{B. Kostant} and \textbf{S. Sternberg} \textit{Hamiltonian group actions and dynamical systems of Calogero type}, Comm. Pure Appl. Math., 31, 481-507 (1978).

\bibitem[S1]{4}
\textbf{R. Steinberg} \textit{Invariants of finite reflection groups}, Canadian Journal of Mathematics, 12, 616-618 (1960).

\bibitem[S2]{8}
\textbf{R. Steinberg} \textit{Differential equations invariant under finite reflection groups}, Transactions of the American Mathematical Society, 112, 392-400 (1964).

\bibitem[W]{2}
\textbf{G. Wilson}, \textit{Collisions of Calogero-Moser particles and an adelic Grassmannian}, Invent. Math 133, 1-41 (1998).

\end{thebibliography}

\end{document}